\documentclass[10pt]{amsart}%
\usepackage{amsmath}
\usepackage{amsfonts}
\usepackage{amssymb}
\usepackage{graphicx}
\usepackage{color}%
\providecommand{\U}[1]{\protect\rule{.1in}{.1in}}
\providecommand{\U}[1]{\protect\rule{.1in}{.1in}}

\theoremstyle{plain}
\numberwithin{equation}{section}
\newtheorem{theorem}{Theorem}[section]
\newtheorem{lemma}[theorem]{Lemma}
\newtheorem{proposition}[theorem]{Proposition}
\newtheorem{corollary}[theorem]{Corollary}

\newtheorem{claim}[theorem]{Claim}
\theoremstyle{definition}
\newtheorem{definition}{Definition}[section]
\theoremstyle{remark}
\newtheorem*{claim*}{Claim}
\newtheorem*{example*}{Example}
\newtheorem*{remark*}{Remark}
\newtheorem{remark}{Remark}[section]

\DeclareMathOperator{\vol}{vol}

\begin{document}
\title[Calibrating optimal transportation]{Pseudo-Riemannian geometry calibrates optimal transportation}
\author{Young-Heon Kim, Robert J. McCann and Micah Warren}
\address{YHK:Department of Mathematics, University of British Columbia, Vancouver BC
Canada and Institute for Advanced Study, Princeton NJ USA. \\
RJM:  Department of Mathematics, University of Toronto, Toronto, Ontario
Canada. MW:  Department of Mathematics, Princeton University, Princeton NJ, USA}
\email{yhkim@math.ubc.ca,mccann@math.toronto.edu,mww@princeton.edu}
\thanks{R.J.M.'s research was supported in part by Natural Sciences and Engineering
Research Council of Canada Grant 217006-08 and United States National Science
Foundation Grant DMS-0354729. Y-H.K. is supported partly by NSF Grant
DMS-0635607 (through the membership at Institute for Advanced Study) and NSERC
discovery grant 371642-09. \ \ M.W. is supported in part by NSF Grant
DMS-0901644. \ Any opinions, findings and conclusions or recommendations
expressed in this material are those of authors and do not reflect the views
of either Natural Sciences and Engineering Research Council of Canada or the
US National Science Foundation. }
\thanks{\indent\copyright 2010 by the authors. }
\date{\today }
\maketitle

\begin{abstract}
Given a transportation cost $c:M\times\bar{M}\rightarrow\mathbf{R}$, optimal
maps minimize the total cost of moving masses from $M$ to $\bar{M}$. We find,
explicitly, a pseudo-metric and a calibration form on $M\times\bar{M}$ such
that the graph of an optimal map is a calibrated maximal submanifold. We
define the mass of space-like currents in spaces with indefinite metrics.

\end{abstract}

\section{Introduction}

The aim of this article is to adapt the notion of calibration (see \cite{HL}
\cite{Me}) to a pseudo-Riemannian framework constructed to describe and
explore the geometry of optimal transportation from a new perspective.

Given a smooth function $c:M\times\bar{M}\rightarrow\mathbf{R}$ (called the
transportation cost), and probability densities $\rho$ and $\bar{\rho}$ on two
manifolds $M$ and $\bar{M}$ (possibly with boundary), a natural variational
problem is to find an optimal map $F:M\rightarrow\bar{M}$ that minimizes the
total cost%
\begin{equation}
\int_{M}c(x,F(x))\rho(x)dx \label{Monge}%
\end{equation}
under the constraint that for every measurable function $f$ on $\bar{M}$,
\begin{equation}
\int_{M}f(F(x))\rho(x)dx=\int_{\bar{M}}f(\bar{x})\bar{\rho}(\bar{x})d\bar{x}.
\label{push-forward}%
\end{equation}
The last condition will be denoted by $F_{\#}\rho=\bar{\rho}$, in which case
we say $F$ \emph{pushes} $\rho$ \emph{ forward }to $\bar{\rho}$. This
variational problem, called optimal transportation, dates back to Monge in
18th century and is currently undergoing a rapid and broad development,
especially in relation to geometry; see Villani's recent book and the
references it contains \cite{Villani}. {}

We introduce a metric on $M\times\bar{M}$ via the following symmetric bilinear
form on $TM\times T\bar{M}$, $\ $defined on a coordinate frame:%
\begin{equation}
h_{c}^{\rho,\bar{\rho}}=\left(  \frac{\rho(x)\bar{\rho}(\bar{x}))}{\left\vert
\det(D\bar{D}c(x,\bar{x}))\right\vert }\right)  ^{\frac{1}{n}}\left(
\begin{array}
[c]{cc}%
0 & -D\bar{D}c\\
-\left(  D\bar{D}c\right)  ^{T} & 0
\end{array}
\right)  \label{metric}%
\end{equation}
and calibration

{%
\begin{equation}
\Phi(x,\bar{x})=\frac{\rho(x)dx+\bar{\rho}(\bar{x})d\bar{x}}{2}.\label{cali}%
\end{equation}
}

Here $D\bar{D}c$ denotes second derivatives of the cost function in mixed
directions. Our central result, Theorem \ref{th:main} can be stated
imprecisely as follows:

\begin{theorem}
\label{thm:imprecise}Under reasonable conditions, the graph of the optimal map
is calibrated by $\Phi,$ and is thus a stable maximal surface with respect
to\ the pseudo-metric $h_{c}^{\rho,\bar{\rho}}$ on $M\times\bar{M}$. \ \ 
\end{theorem}

A new geometric aspect of optimal transportation was observed by the first two
authors \cite{KM}. Namely, the transportation cost $c$ induces a certain
pseudo-metric on $M\times\bar{M}$, \eqref{eq:h}, in such a way that positivity
of its Riemannian curvature tensor on certain sections gives a necessary
condition for the regularity of general optimal maps $F$. This result gives a
geometric perspective on the fundamental regularity theory developed for
optimal maps by Ma, Trudinger and Wang \cite{MTW} \cite{TW} and Loeper
\cite{Loeper}. Moreover, the graph of the optimal map $F$ is a Lagrangian
submanifold with respect to the K\"{a}hler form of the pseudo-metric (see
Section~\ref{S:optimal} or \cite{KM} for more details).

Mealy \cite{Me} developed the idea of calibration on pseudo-Riemannian
manifolds, and introduced the special Lagrangian equations for signature
$(n,n)\ $metrics:
\[
\sum_{k=0}^{\lfloor(n-1)/2\rfloor}\sigma_{2k+1}(D^{2}\phi)=0\ \ \text{ (here,
$\sigma_{k}$ denotes the $k$-symmetric function).}%
\]
These describe functions whose gradient graphs $(x,\nabla\phi(x))$ are maximal
surfaces in the \textquotedblleft standard" pseudo-Euclidean space, $(%
\mathbb{R}
^{n+n},dx^{2}-dy^{2}),$ \ provided that
\[
\left\vert D^{2}\phi\right\vert <1
\]
which is equivalent to the graph being spacelike. \ After Hitchin \cite[Prop.
3]{Hitchin} connected Monge-Amp\`{e}re equations and special Lagrangian
submanifolds in a pseudo-Euclidean space, the third author \cite{Warren}
explored these equations, which essentially describe Mealy's submanifolds
under a coordinate change. \ In Hitchin's coordinate setting the gradient map
of a convex function $\phi$ solving the standard Monge-Amp\`{e}re equation
\[
\det D^{2}\phi=1,
\]
gives a calibrated submanifold in the product space $\mathbf{R}^{n}%
\times\mathbf{R}^{n}$ equipped with a pseudo-Euclidean metric, $dx\cdot dy.$
\ \ In fact,  Hitchin's definition of special Lagrangian is that a linear
combination of certain volume forms vanish\ along a Lagrangian submanifold.
\ Our calibration (\ref{cali}) encodes precisely this definition for the
volume forms $\rho dx$ and $\bar{\rho}d\bar{x}.$

In this paper, we provide a framework that combines and extends these results
by providing the pseudo-metric defined in ~(\ref{metric}), which is
conformally equivalent to the pseudo-metric in \cite{KM} but depends on $\rho$
and $\bar{\rho}$ in addition to $c$. With respect to this metric, the graphs
of optimal maps are calibrated, thus give maximal submanifolds (see
Theorem~\ref{th:main} and Corollary~\ref{cor:mean curv}). This demonstrates
how the functional extremality of optimal maps for
\eqref{Monge}--\eqref{push-forward} characterizes in a natural way the
geometric extremality of their graphs.

Mealy's calibrations compare admissible (spacelike, in our case of interest)
smooth submanifolds to other admissible smooth submanifolds, but we would like
a more general setting allowing for\ spacelike submanifolds which are merely
rectifiable. \ \ 

Solutions to the optimal transportation problem are not in general smooth.
Allowing for nonsmooth solutions, we will use the language of currents, which
requires that we verify that the submanifolds in question are rectifiable.
\ This has been shown recently in \cite{MPW}, given assumption (A2), using the
Lewy transformation, which returns us to a situation which locally looks like
Mealy's,\ where the submanifolds in question are automatically Lipschitz
graphs. The development of currents generalizes some results for the Euclidean
cost function appearing in \cite{AA}.\ \ The behavior of solutions to the
problem of optimal transportation for general measures on general metric
spaces can be quite wild. It is unclear how much of this theory can be adapted
to more general cases.

\ Our use of currents will require some apparently new definitions of mass for
currents in a pseudo-Riemannian manifold, see definition~\ref{def:comass}%
.\ While the corresponding notions of mass from Euclidean geometric measure
theory are expressed in terms of suprema, we will need to\ express mass in
terms of infima, and will also need some sort of space orientation for these
definitions to make sense. These definitions recover the volume of a smooth
spacelike submanifold.

After laying down some pseudo-Riemannian geometric measure theory in section
2, we will define calibrations on spaces with pseudo-metrics. \ In the final
section we show how the optimal transportation problem fits into this setting.

The calibrated submanifolds with the metric we are considering are not in
general graphs of the optimal maps, rather these are graphs of Lie solutions
as developed by Delanoe \cite{D} . \ The moduli space of such calibrated
submanifolds is discussed in \cite{W1}\ .

The authors are thankful to Reese Harvey and the referee, who pointed out the
reference \cite{Me} as well as other helpful suggestions, also to Blaine
Lawson for his interest and helpful discussions, and to Christina Sormani for
helpful suggestions. \ 

\section{Spacelike Currents in an Indefinite (pseudo) Metric}

We formulate some definitions which adapt the geometric measure theoretic
notion of \emph{mass} to oriented manifolds with indefinite metric. This
notion will allow us to compare the mass of calibrated currents to homologous
currents in Section~\ref{S:optimal}. \ 

Let $(N^{n,m},h,\tau)$ denote a smooth $(n+m)$-dimensional manifold, a metric
with signature $(n,m)$ and a space orientation $n$-form $\tau.$ \ An oriented
$n$-plane $P$ is spacelike {if $h|_{P}>0,$ and is $\tau$-oriented if }%
$\tau(\xi)>0$ for any $n$-vector $\xi\in\Lambda^{n}T_{x}N$ representing $P.$
\ \ Let $G_{+}(n,T_{x}N)$ denote the collection of $n$-vectors which are both
spacelike and {$\tau$-oriented. (To be clear, }$G_{+}${ is a cone, not a
Grassmannian.)\ \ \ Note that }the \textquotedblleft unit sphere" of
$n$-vectors in an indefinite metric will have distinct connected components,
{\ since each such plane comes with two orientations distinguished by the sign
of the space orientation form $\tau$} (see
Proposition~\ref{prop:space orientation disconnected}). \ Here, as in many
other of our arguments, it is crucial that $n$ is the maximum dimension of any
spacelike plane. \ An element $\xi\in\Lambda^{n}T_{x}N$ $\ $is called
\emph{simple} if it is a single product of $n$ vectors, i.e. $\xi=v_{1}%
\wedge\ldots\wedge v_{n}.$ \ \ For a simple $\xi\in G_{+}(n,T_{x}N)$ we will
be interested in {the $h$-norm of }$\xi${, which is the} $n$-dimensional
$h$-volume defined by
\begin{equation}
\Vert\xi\Vert_{h}:=\sqrt{\det(h(v_{i},v_{j}))_{1\leq i,j\leq n}}.
\label{n-volume}%
\end{equation}

In the following, a \emph{current} shall mean a de Rham $n$-current with
compact support on $N,$ that is, an element of the dual space $(\Omega
^{n}N)^{\ast}$ of the space $\Omega^{n}N$ of smooth $n$-forms on
$N$.\ \ Recall that a compact oriented $n$-submanifold with locally finite
Hausdorff measure defines a linear functional on the space of $n$-forms by
integration, thus a "reasonable" nonsmooth submanifold is also a current, in a
similar way. The support of current $T$ is given by the following: \ Let $U$
be the largest open set having the property that $T(\varphi)=0$ whenever
$\varphi$ is compactly supported inside $U$. \ The \emph{support} of $T,$
spt$(T),$ is the complement of $U.$ \ \ 

\begin{definition}
[Mass of a current; comass of a form]\label{def:comass} Define the set of
simple $\tau$-\emph{oriented} space-like unit $n$-planes by
\[
\mathcal{P}_{x}=\ \{\xi\in\Lambda^{n}T_{x}N|\text{ }\xi=v_{1}\wedge
\ldots\wedge v_{n}\text{, }\tau(\xi)>0,h|_{span(v_{1},\ldots,v_{n}%
)}>0,\left\Vert \xi\right\Vert _{h}=1\}.
\]
Define the \emph{oriented comass} of an $n$-form $\psi_{x}$ at a point via
\[
\left\Vert \psi_{x}\right\Vert _{h}^{\ast}=\inf_{\xi\in\mathcal{P}_{x}}%
\psi_{x}(\xi).
\]
On any set $U\subset N,$ define the \emph{oriented comass }of a $n$-form
$\psi$ on $U$ as%
\[
\left\Vert \psi\right\Vert _{(U,h)}^{\ast}=\inf_{x\in U}\left\Vert \psi
_{x}\right\Vert _{h}^{\ast}.
\]
Now define the \emph{oriented} \emph{mass} of a current $T\in$ $(\Omega
^{n}U)^{\ast}$ by
\[
\left\Vert T\right\Vert _{(U,h)}=\inf_{\left\Vert \psi\right\Vert
_{(U,h)}^{\ast}\geq1}T(\psi).
\]

\end{definition}

Some remarks about the definitions: \ First, since the space $\mathcal{P}_{x}$
is noncompact, we observe that the values of all of the infima in
Definition~\ref{def:comass} may be $-\infty.$ \ Time-like or negatively
oriented planes are given infinite weight, thus any current with enough
time-like or negatively oriented planes is given a mass of $-\infty$.
\ Fortunately, this rules out certain pathologies that occur for
pseudo-metrics (see examples in \cite[section 4]{Warren}), and recovers
expected values of mass for smooth space-like sets.

For Riemannian (thus positive definite) metrics, a calibration is a closed
$p$-form $\Psi$ such that for all $p$-vectors $\xi,$ $\Psi_{x}(\xi
)\leq\left\Vert \xi\right\Vert $ \ (see \cite{HL}). \ A calibrated current $S$
is one for which $S(\Psi)=\left\Vert S\right\Vert .$ \ It follows from Stokes'
theorem that smooth calibrated manifolds are minimal. \ In a pseudometric, a
calibration needs to give the opposite inequality. To be precise, we define:

\begin{definition}
[Calibration for indefinite metric]\label{def:calibration} A calibration on
$(N^{n,m},h,\tau)$ is an $n$-form $\Psi$ such that
\[
d\Psi=0
\]
and
\[
\Psi_{x}(\xi)\geq\left\Vert \xi\right\Vert _{h}%
\]
at each point $x,$ for each space-like tangent plane $\xi.$ Notice that this
latter condition is equivalent to $\Vert\Psi\Vert_{h}^{\ast}\geq1$.
\end{definition}

\begin{proposition}
Let $S$ be a smooth, oriented, compact $n$-dimensional space-like submanifold
{of $N$}. Let $\vol_{h}(S)$ denote the volume of $S$ with respect to the
pseudo-metric $h$. Then, $\vol_{h}(S)=\Vert S\Vert_{(N,h)}$.
\end{proposition}

\begin{proof}
At each $x,$ for a given spacelike {simple $n$-vector $\xi\in\Lambda^{n}%
T_{x}N$}, the form given by the $h$-inner product with $\xi$ has comass $1$
(See \cite[pg 797]{Me} ). This form depends smoothly on the tangent space to
$S$. {(The corresponding simple $n$-vector field is obtained by choosing a
smooth frame field and taking the $n$-product of the basis vectors.)} \ Thus
we can find a form $\psi\in\Omega^{n}S$ such that for each $x\in S$, $\psi
_{x}(\varsigma)\geq\Vert\varsigma\Vert_{h}$ for all $\varsigma\in
\mathcal{P}_{x}$ with equality when and only when $\varsigma$ represents the
oriented tangent plane to $S.$ \ Extend this smoothly to a form $\psi_{0}$
which has positive comass on a neighborhood $V_{0}$ of $S$. \ \ Dividing by
the pointwise comass at each point where the comass is positive, we have that
$\left\Vert \psi_{0}\right\Vert _{(V_{0},h)}^{\ast}=1.$ Next, take a cover of
$N$ with open sets $V_{i,}$ so that $V_{i}\cap S=\emptyset,$ for $i\geq1,$ and
so that on each $V_{i,}$ there is a $\psi_{i}$ with $\left\Vert \psi
_{i}\right\Vert _{\left(  V_{i},h\right)  }^{\ast}\geq1$. \ Such $V_{i\text{
}}$ are easily locally available. Noting that the comass is superadditive, we
may sum over a partition of unity and get a form, which we call $\psi$, with
$\left\Vert \psi\right\Vert _{\left(  N,h\right)  }^{\ast}\geq1.$ Also, for
any $\varphi$ with $\left\Vert \varphi\right\Vert _{(N,h)}^{\ast}\geq1,$
$S(\varphi)\geq S(\psi),$ by the choice of $\psi.$ It follows that {%
\[
\inf_{\left\Vert \varphi\right\Vert _{(N,h)}^{\ast}\geq1}S(\varphi
)=S(\psi)=\int_{S}\psi=\int_{S}d\vol_{h}(S).
\]
}
This completes the proof.
\end{proof}

From the above proposition, it is legitimate to use the notion of mass of
currents to compare volumes of smooth space-like submanifolds.

We conclude this section with a remark:

\begin{proposition}
\label{prop:space orientation disconnected} \bigskip Let $h$ be an indefinite
metric on $%
\mathbb{R}
^{n+m\text{ }}$ of the form $dx^{2}-dy^{2}.$ The \textquotedblleft unit
sphere" $\Vert\xi\Vert_{h}=1$ of space-like $n$-planes is disconnected. \ 
\end{proposition}

\begin{proof}
Let $\xi$ be an $n$-plane given by $v_{1}\wedge\cdots$ $\wedge v_{n}.$ \ The
projection $\pi$ of the independent set $\{v_{i}\}$ onto $%
\mathbb{R}
^{n}$ gives an independent set, otherwise the plane would contain time-like or
null vectors. \ Define $\tau(P)=\det(\pi(v_{1}\wedge\cdots$ $\wedge v_{n}))$.
$\ \ $It is clear that the alternating form $\tau~$will be either positive or
negative on any space-like plane.
\end{proof}

\section{\bigskip Application to Optimal Transportation}

\label{S:optimal} { Let $M$, $\bar{M}$ be oriented $n$-dimensional manifolds
and let $\rho$, $\bar{\rho}$ be smooth volume forms on $M$, $\bar{M}$,
respectively, with $\int_{M}\rho=\int_{\bar{M}}\bar{\rho}=1$. In oriented
local coordinates $(x^{1},\cdots,x^{n},\bar{x}^{1},\cdots\bar{x}^{n})$ of
$(x,\bar{x})\in M\times\bar{M}$, we have the expression $\rho=\rho(x)dx$,
$\bar{\rho}=\bar{\rho}(\bar{x})d\bar{x}$, $\rho(x),\bar{\rho}(\bar{x})>0$,
where $dx=dx^{1}\wedge\cdots\wedge dx^{n}$, $d\bar{x}=d\bar{x}^{1}\wedge
\cdots\wedge d\bar{x}^{n}$. Define an $n$-form $\Phi$ on $M\times\bar{M}$ in
local coordinates (see upcoming subsection \ref{S:orientation}) by
\begin{equation}
\Phi(x,\bar{x})=\frac{\rho(x)dx+\bar{\rho}(\bar{x})d\bar{x}}{2}.
\label{def:cal}%
\end{equation}
Notice that we could write this in the invariant form }%
\[
\Phi=\frac{\pi^{\ast}\rho+\bar{\pi}^{\ast}\bar{\rho}^{\ast}}{2}%
\]
where { }$\pi^{\ast},\bar{\pi}^{\ast}$ are pullbacks of the projection mapping
to $M$ and $\bar{M}$ respectively. \ In this section, we show the compactly
supported portions of the graph of an optimal map are calibrated by the form
$\Phi$, hence maximal with respect to the pseudo-metric $h^{\rho,\bar{\rho}}$
in metric.

{Let $c:M\times\bar{M}\rightarrow\mathbf{R}$ be a positive, continuous,
superdifferentiable cost function which is smooth on a set }${N=}${$\left(
M\times\bar{M}\right)  $-$\mathfrak{C}$}, \ where{ $\mathfrak{C}$ is a closed
measure zero set which we will call the \textquotedblleft cut locus". (The
reason for the terminology \textquotedblleft cut-locus" is clear if the cost
function is the distance squared function on a Riemannian manifold.) For some
regularity issues later, we further assume that the cost function is uniformly
semi-concave on compact sets, i.e. for any compact coordinate charts
$K\times\bar{K}\subset M\times\bar{M}$, there exists a smooth function
$f:{K\times\bar{K}}\rightarrow\mathbf{R}$ such that $f+c$ is concave on
$K\times\bar{K}.$} \ \ Let $D\bar{D}c$ be the $n\times n$ matrix\ given by
\[
\left(  D\bar{D}c\right)  _{i\bar{j}}(x,\bar{x})=\frac{\partial^{2}}{\partial
x^{i}\partial\bar{x}^{j}}c(x,\bar{x}).
\]
On $N$ we will require that
\begin{equation}
\det\left(  D\bar{D}c\right)  \neq0 \tag{A2}%
\end{equation}
and that $\bar{c}(\bar{x},x):=c(x,\bar{x})$ \ both satisfy the \emph{twist
}(A1) condition: \ For each $x,$ \
\begin{equation}
\bar{x}\rightarrow Dc(x,\bar{x}) \tag{A1}%
\end{equation}
is invertible, with an inverse depending continuously on $x$, and vice versa. \ 

\ Now let $h_{c}$ be the pseudo-metric on $N\subset M\times\bar{M}$ that is
defined at each $T_{(x,\bar{x})}N=T_{x}M\times T_{\bar{x}}\bar{M}$ as
\[
h_{c}(\partial x^{i},\partial\bar{x}^{j})=h_{c}(\partial\bar{x}^{j},\partial
x^{i})=\frac{\partial^{2}}{\partial x^{i}\partial\bar{x}^{j}}c(x,\bar{x})
\]
which is represented in coordinates\ by the nondegenerate symmetric matrix
\begin{equation}
h_{c}=\left(
\begin{array}
[c]{cc}%
0 & -D\bar{D}c\\
-\left(  D\bar{D}c\right)  ^{T} & 0
\end{array}
\right)  \label{eq:h}%
\end{equation}
(This metric was introduced by the first and second author in \cite{KM}.)
{Note that as a matrix, $h_{c}$ has $n$ positive and $n$ negative eigenvalues;
in fact, we can choose coordinates so that }$-D\bar{D}c$ is positive definite.

Given assumption (A1) together with some decay conditions on $\rho,\bar{\rho}$
if $M$, $\bar{M}$ are noncompact, we will always have a unique solution (see,
for example, \cite[Theorem 10.28]{Villani}) to the optimal transport problem,
namely, a map $F:M\rightarrow\bar{M}$ which satisfies $F_{\#}\rho=\bar{\rho}$
in the sense of measures. \ One also has the minimizing measure $\gamma$ (c.f
\cite[Theorem 5.10]{Villani}) on $M\times\bar{M}$\ which has marginals
$\rho,\bar{\rho},$ which is supported on the set $\Gamma$ which is the graph
of the (possibly multivalued) optimal map. \ The measure $\gamma$ solves the
Kantorovich problem, producing two potential functions $u$ $\in C(M)$ and $v$
$\in C(\bar{M})$. From these potentials we may derive the optimal mapping
$F_{u}:$\ $M\rightarrow\bar{M}$ and an inverse mapping $F_{v}:\bar{M}$
$\rightarrow M.$ \ Note that $F_{v}$ solves the symmetric optimal
transportation problem. \ { We have the following now well-known fact: }

\begin{proposition}
(c.f. \cite{GM}) Both $F_{u}$ and $F_{v}$ differentiable almost everywhere.
\end{proposition}

\noindent Here is a sketch of the proof: The cost function $c$ is
semi-concave, which gives semi-convexity of the $c$-convex solution $u.$
(Using the conventions in \cite[Theorem 5.2]{Villani}.)\ Then by Alexandrov's
theorem, $u$ is twice differentiable almost everywhere, and together with
assumptions (A1) (A2) we see that the transport maps must be differentiable
where $u$ is so.

\subsection{Choosing orientation and coordinates}

\label{S:orientation}

We are assuming that $M$ comes with an orientation, which we respect. \ At any
point $(x,\bar{x})\in$ $N,$ $D\bar{D}c$ is nondegenerate, so by a change of
coordinates (only on $\bar{M}$) we can arrange that $-D\bar{D}c$ is positive
definite. \ \ This locally fixes an orientation on $\bar{M}$, which applies
locally in $N.$ \ In particular, it is possible that the orientations on the
same neighborhoods on $\bar{M}$ may differ, depending on where they lie with
respect to the product manifold. \ Fortunately, it is easy to check that on
connected components of $N,$ the (A2)~condition implies that our choice of
coordinates on $\bar{M}$, and hence orientation, is consistent. \ 

\begin{claim}
Choosing coordinates on $\bar{M}$ so that $-D\bar{D}c+\left(  -D\bar
{D}c\right)  ^{T}$ is positive definite, the locally defined forms
\[
\tau=\frac{dx+d\bar{x}}{2}%
\]
are consistent space orientation forms on $N.$
\end{claim}

\begin{proof}
Let $(x,\bar{x})$ $\subset\Omega\times\bar{\Omega}$ be any product coordinate
chart, with $x$ consistent with the orientation on $M.$ \ Let $\bar{x}%
=\varphi(\bar{y})$ for some diffeomorphism $\varphi:\bar{B}\rightarrow
\bar{\Omega}\subset\bar{M}.$ \ Then
\[
\partial_{y_{j}}\partial_{x_{i}}c(x,\varphi(\bar{y}))=\partial_{\bar{x}_{k}%
}\partial_{x_{i}}c(x,\bar{x})\partial_{y_{j}}\varphi^{\bar{k}}.
\]
At a point, we are free to choose $\partial_{y_{j}}\varphi^{\bar{k}}$ to be
anything (nondegenerate)\ that we want, so we choose it to be the inverse of
$-\partial_{\bar{x}_{k}}\partial_{x_{i}}c(x,\bar{x}).$ \ By smoothness of $c$
on $N,$ $-\partial_{y_{j}}\partial_{x_{i}}c(x,\varphi(\bar{y}))$ will remain
(after symmetrizing) positive definite on a nieghborhood on $N,$ and hence on
some subset of $\bar{A}\subset\bar{B}.$ \ We declare the coordinates on
$\bar{A}$ to be oriented, and $\Omega\times\bar{A}$ defines an orientation on
$N.$\ One can repeat this construction at every point on $N$ and take a
locally finite cover of such neighborhoods. \ These are consistent on any
overlap, as both charts are defined to have differential close to the inverse
of a smoothly changing matrix.
\end{proof}

At this point, we fix a space orientation form $\tau$ by summing over a
partition of unity. \ \ We also note that the coordinate definition
\ref{def:cal} of $\Phi$ is now justified. \ \ For future reference, we note
that
\begin{equation}
\det(-\bar{D}Dc)>0.\label{orientation determinant}%
\end{equation}

\subsection{\bigskip What happens where the transport map is smooth}

Where differentiable, the potential $u$ satisfies%
\[
Du(x)=-Dc(x,F_{u}(x))
\]
(we are using conventions as in \cite[Theorem 5.2]{Villani})\ and
\[
D^{2}u(x)=-D^{2}c(x,F_{u}(x))-D\bar{D}c(x,F_{u}(x))DF_{u}(x)
\]
in particular
\begin{equation}
-D\bar{D}c(x,F_{u}(x))DF_{u}(x)=D^{2}u(x)+D^{2}c(x,F_{u}(x))
\label{eq:graph is Lagrangian}%
\end{equation}
the right-hand side of which is clearly symmetric. \ Further, $c$-convexity
implies that
\begin{equation}
D^{2}u(x)\geq-D^{2}c(x,F_{u}(x)). \label{eq:graph is spacelike}%
\end{equation}
To satisfy $F_{\#}\rho=\bar{\rho},$ the map $F_{u}$ must also satisfy, where
differentiable
\begin{equation}
\bar{\rho}(F_{u}(x))\det DF_{u}(x)=\rho(x) \label{eq:measure preserving}%
\end{equation}
recalling, (\ref{orientation determinant}). \ Now consider the following
symplectic form on $M\times\bar{M}$ :
\[
\omega_{c}=\left(
\begin{array}
[c]{cc}%
0 & -D\bar{D}c\\
\left(  D\bar{D}c\right)  ^{T} & 0
\end{array}
\right)  ,
\]
and the following conformal perturbation of $h_{c}$:
\begin{equation}
h^{\rho,\bar{\rho}}=\frac{1}{2}\left(  \frac{\rho(x)\bar{\rho}(\bar{x}%
)}{\left\vert \det(-D\bar{D}c)\right\vert }\right)  ^{\frac{1}{n}}h_{c}.
\label{eq:conformal h}%
\end{equation}
First we show that the graph of $F_{u}$ is Lagrangian, wherever it is
differentiable, by the following. Pull-back the form $\omega_{c}$ to $M,$ and
evaluate on any two tangent vectors
\begin{align*}
(Id\times F)^{\ast}\omega_{c}(\partial_{i},\partial_{j})  &  =\langle-D\bar
{D}c\left(  \partial_{i}F_{u}\right)  ,\partial_{j}\rangle+\langle\left(
D\bar{D}c\right)  ^{T}\partial_{i},\left(  \partial_{j}F_{u}\right)  \rangle\\
&  =\langle-D\bar{D}c\left(  \partial_{i}F_{u}\right)  ,\partial_{j}%
\rangle+\langle\partial_{i},D\bar{D}c\left(  \partial_{j}F_{u}\right)
\rangle\\
&  =\left(  D\bar{D}cDF\right)  _{ij}-\left(  D\bar{D}cDF\right)  _{ji}%
\end{align*}
which vanishes by~(\ref{eq:graph is Lagrangian}).

Similarly, pulling back the metric $h_{c}:$%
\begin{align*}
(Id\times F)^{\ast}h_{c}(\partial_{i},\partial_{j})  &  =\langle-D\bar
{D}c\left(  \partial_{i}F_{u}\right)  ,\partial_{j}\rangle+\langle\left(
-D\bar{D}c\right)  ^{T}\partial_{i},\left(  \partial_{j}F_{u}\right)
\rangle\\
&  =-\left(  D\bar{D}cDF\right)  _{ij}-\left(  D\bar{D}cDF\right)  _{ji}%
\end{align*}
which is nonnegative\ by (\ref{eq:graph is spacelike}). The measure preserving
condition (\ref{eq:measure preserving}), along with
~(\ref{eq:graph is spacelike}) and (\ref{eq:graph is Lagrangian}) guarantee
that this metric will be strictly positive wherever the map is differentiable.

\begin{proposition}
\label{prop:calibrated}

Let $F:M\rightarrow\bar{M}$. Assume that $F$ is $c$-monotone, i.e.
\begin{equation}
c(x,F(x))+c(y,F(y))-c(x,F(y))-c(y,F(x))\leq0\ \ \text{ for all $x,y\in M$.}
\label{2monotone}%
\end{equation}
\ At any point of differentiability of $F,$ the following holds%
\[
(Id\times F)^{\ast}\Phi(\partial_{1},...,\partial_{n})(x)\geq\sqrt{\det
g_{ij}(x)}%
\]
where
\[
g_{ij}=(Id\times F)^{\ast}h_{c}^{\rho,\bar{\rho}}(\partial_{i},\partial_{j})
\]
is the induced metric, with equality if and only if \ both
(\ref{eq:measure preserving}) and (\ref{eq:graph is Lagrangian}) hold, in
which case%
\[
(Id\times F)^{\ast}\Phi(\partial_{1},...,\partial_{n})(x)=\sqrt{\det
g_{ij}(x)}=\rho(x).
\]

\end{proposition}

\begin{proof}
\bigskip First we compute the calibration, from (\ref{def:cal}) \ \
\[
(Id\times F)^{\ast}\Phi(\partial_{1},...,\partial_{n})=\frac{\rho+\det
DF\bar{\rho}(F(x))}{2}.
\]
Next, using (\ref{eq:conformal h})
\begin{align*}
g_{ij}  &  =\frac{1}{2}\left(  \frac{\rho(x)\bar{\rho}(F(x))}{\left\vert
\det(D\bar{D}c)\right\vert }\right)  ^{\frac{1}{n}}\left(
\begin{array}
[c]{cc}%
I & \left(  DF\right)  ^{T}%
\end{array}
\right)  \left(
\begin{array}
[c]{cc}%
0 & -D\bar{D}c\\
-\left(  D\bar{D}c\right)  ^{T} & 0
\end{array}
\right)  \left(
\begin{array}
[c]{c}%
I\\
DF
\end{array}
\right) \\
&  =\left(  \frac{\rho(x)\bar{\rho}(F(x))}{\left\vert \det(D\bar
{D}c)\right\vert }\right)  ^{\frac{1}{n}}\left(  \frac{\left(  DF\right)
^{T}\left(  -D\bar{D}c\right)  ^{T}+\left(  -D\bar{D}c\right)  \left(
DF\right)  }{2}\right)  .
\end{align*}
Now taking the determinant
\[
\det g_{ij}\leq\left(  \frac{\rho(x)\bar{\rho}(F(x))}{\left\vert \det(D\bar
{D}c)\right\vert }\right)  \det\left(  \left(  DF\right)  ^{T}\left(
-D\bar{D}c\right)  ^{T}\right)
\]
with equality if and only if $\left(  DF\right)  ^{T}\left(  -D\bar
{D}c\right)  ^{T}$ is symmetric, recalling the fact\ (see \cite{Warren}, Lemma
3.1, also for a different formulation \cite[pg 803]{Me} ) that for each
$n\times n$ matrix $B$ with property $\langle Bv,v\rangle\geq0$ for every
$v$,
\[
\det(\frac{1}{2}(B+B^{T}))\leq\det B
\]
and the equality holds if and only if $B=B^{T}$. \ \ We show in Lemma
\ref{prop:nonnegative mass} that due to $c$-monotonicity of $F$, the matrix
$B=\left(  DF\right)  ^{T}\left(  -D\bar{D}c\right)  ^{T}$ satisfies $\langle
Bv,v\rangle\geq0$ for every $v$. \ 

In particular, applying Cauchy-Schwarz \
\[
\sqrt{\det g_{ij}(x)}\leq\sqrt{\rho(x)\bar{\rho}(F(x))\det\left(  DF\right)
}\leq\frac{\rho(x)+\bar{\rho}(F(x))\det\left(  DF\right)  }{2}%
\]
with equality if and only if (\ref{eq:graph is Lagrangian}) and
(\ref{eq:measure preserving}) hold, respectively. \ 
\end{proof}

\begin{corollary}
\label{cor:cal}The form $\Phi$ is a calibration. \ 
\end{corollary}

\begin{proof}
The above calculation shows that this form has comass $1,$ noting that every
spacelike plane can be obtained as the tangent space to the graph of a
monotone function. \ Closedness of the form is clear. \ 
\end{proof}

\subsection{Solutions defining a current}

For a given optimal transport map $F$ (possible multivalued) we deal with the
closure of the graph of $F,$ that is
\[
\Gamma=\text{cl}\left(  Id\times F\right)  (M)\subset M\times\bar{M}.
\]
{ It is well known that the closure $\Gamma$ of any such set satisfies
$c$-cyclical monotonicity: if $(x_{i},\bar{x}_{i})\in{\Gamma},i=1,\cdots,l$
then
\[
\sum_{i}c(x_{i},\bar{x}_{i})\leq\sum_{i}c\left(  x_{i},\bar{y}_{\sigma
(i)}\right)  \ \ \text{for any permutation $\sigma$ of $\{1,\cdots,l\}$.}%
\]
In particular, $F$ satisfies $c$-monotonicity.} \ In our case we must deal
with {$\Gamma$} avoiding the cut-locus:\ Our arguments require that the
current we are defining is compactly supported inside the region where the
metric is smooth. If the metric is not smooth it may be difficult to define
the comass\ of an $n$-form. Our main goal of this subsection is the following,
which is similar in nature to the result in \cite{MPW}.

\begin{proposition}
Suppose that $c$ is smooth, and satisfies (A1) (A2) on a neighborhood of
{$\Gamma$}. Then {$\Gamma$} defines a current in a natural way. \ 
\end{proposition}

Proof: \ We begin with the following, due to \cite{MPW}. \ 

\begin{proposition}
Let $F$ be an optimal transportation map (possible multivalued)\ and suppose
that c is smooth and satisfies (A2)\ in a neighborhood of the image
\[
\Gamma=\text{cl}\left(  Id\times F\right)  (M)\subset M\times\bar{M}.
\]
Then the image is $n$-rectifiable and locally finite.
\end{proposition}

The proof of this proposition is contained in \cite{MPW}. \ \ We give the idea
here, as we need the construction to define the current. \ Near any point
$\left(  x,\bar{x}\right)  $ in $\Gamma$ one can choose coordinates so that
$D\bar{D}c=-I_{n}$ and $\left(  x,\bar{x}\right)  $ is the origin. \ Near
$\left(  x,\bar{x}\right)  $ \ expand $c$ as a second order Taylor polynomial
plus third order remainder. \ From the Taylor expansion together with
$c$-monotonicity (\ref{2monotone})
\[
c(x,\bar{x})+c(0,0)\leq c(x,0)+c(0,\bar{x})
\]
one can show that
\[
x\cdot\bar{x}\geq O(3)
\]
so that for a small enough neighborhood $U$,
\[
\left\vert x+\bar{x}\right\vert =\sqrt{|x|^{2}+|\bar{x}|^{2}+2x\cdot\bar{x}%
}\geq\sqrt{|x|^{2}+|\bar{x}|^{2}-O(3)}%
\]
in particular the map $G:\Gamma\rightarrow%
\mathbb{R}
^{n}$%
\[
G(x,\bar{x})=x+\bar{x}%
\]
is injective with Lipschitz inverse. \ \ Choose a compact $K\subset U,$ then
\[
E:=G(\Gamma\cap K)\subset%
\mathbb{R}
^{n}\text{ }%
\]
is a closed subset. \ One can compute that $G^{-1\text{ }}$ is Lipschitz on
$E$ and so extends to a Lipschitz map on an open set containing $E.$ \ \ Now
define the currents, first on $%
\mathbb{R}
^{n}$
\[
S=\chi_{E}\vec{T}_{x}%
\mathbb{R}
^{n}%
\]
and then on $M\times\bar{M}$
\[
T_{K}=\left(  G^{-1\text{ }}\right)  _{\#}S.
\]
The result is a current on $M\times\bar{M}$ \ which is supported on the set
$\Gamma\cap K$ and is represented by the tangent space to $\Gamma$ wherever it
exists, which is almost everywhere $d\mathcal{H}^{n}.$

Observe that
\[
G^{\ast}dx=dx+d\bar{x}%
\]
which, for our purposes, is the space orientation form. It follows that the
tangent spaces to $\left(  G^{-1\text{ }}\right)  _{\#}S$ are appropriately
oriented, given the discussion in subsection \ref{S:orientation}.\ In
particular there will be no cancellation when summing over a partition of
unity. \ Given any precompact neighborhood $U~$in $M\times\bar{M}$ we can
cover $\Gamma\cap U$ with finitely many open sets chosen as above. \ Taking a
(locally finite)\ partition of unity, we build the current $T.$

\bigskip

Now that we have an integral current, we can write $T$ as
\begin{equation}
T(\eta)=\int_{{\Gamma}}\langle\eta(x),\xi(x)\rangle d\mathcal{H}^{n}
\label{eq:T}%
\end{equation}
where $\xi(x)$ is an $\mathcal{H}^{n}$ almost everywhere simple unit tangent
$n$-vector to {$\Gamma$}.

\bigskip

\subsection{The current {$T$} is calibrated}

\ Before we continue we need the following,

\begin{lemma}
\label{prop:nonnegative mass} Let $\Gamma\subset M\times\bar{M}$ be any
$c$-monotone set. \ If $V+\bar{V}$ is a vector (decomposed in the obvious way)
in the tangent space to $\Gamma$ at $(x_{0},\bar{x}_{1})$ then
\[
-\partial_{V}\partial_{\bar{V}}c|_{(x,\bar{x})=(x_{0},\bar{x}_{1})}\geq0.
\]
In particular, $V+\bar{V}$ is weakly spacelike. \ 
\end{lemma}

\begin{proof}
Let $S(t)=(s(t),\bar{s}(t))$ be a path in $\Gamma$ with $s(0)=$ $(x_{0}%
,\bar{x}_{1})$ and $s^{\prime}(0)=V+\bar{V}.$ \ By the monotonicity condition
at $t=0$

First,at $t=0$
\begin{align*}
\frac{d}{dt}\left\{  c(s(t),\bar{s}(t))+c(x_{0},\bar{x}_{1})-c(s(t),\bar
{x}_{1})-c(x_{0},\bar{s}(t))\right\}   &  =\\
Dc\cdot s^{\prime}(0)+\bar{D}c\cdot\bar{s}^{\prime}(0)-Dc\cdot s^{\prime
}(0)-D\bar{c}\cdot\bar{s}^{\prime}(0)  &  =0
\end{align*}
and using the monotonicity condition at $t=0$
\begin{multline*}
\frac{d}{dt}\left\{  c(s(t),\bar{s}(t))+c(x_{0},\bar{x}_{1})-c(s(t),\bar
{x}_{1})-c(x_{0},\bar{s}(t))\right\}  =\\
D^{2}c(s^{\prime}(0),s^{\prime}(0))+Dc\cdot s^{\prime\prime}(0)+2D\bar
{D}c(s^{\prime}(0),\bar{s}^{\prime}(0))+\bar{D}^{2}c(\bar{s}^{\prime}%
(0),\bar{s}^{\prime}(0))\\
+\bar{D}c\cdot\bar{s}^{\prime\prime}(0)-D^{2}c(s^{\prime}(0),s^{\prime
}(0))-Dc\cdot s^{\prime\prime}(0)-\bar{D}^{2}c(\bar{s}^{\prime}(0),\bar
{s}^{\prime}(0))-\bar{D}c\cdot\bar{s}^{\prime\prime}(0)\leq0
\end{multline*}
in particular $2D\bar{D}c(V,\bar{V})\leq0.$
\end{proof}

\begin{lemma}
\label{comass negative} Suppose that $\xi$ is an $n$-vector representing the
tangent space to the current $T$ at $x,$ and $\eta$ is a form such that
$\langle\eta_{x},\xi\rangle<0.$ \ Then $\eta$ has negative comass. \ 

\begin{proof}
First, we note by construction, that the tangent planes are $\tau$-oriented.
\ \ \ If $\xi$ is strictly spacelike, the result follows from definition
\ref{def:comass} by scaling $\xi$ to unit size. \ \ \ So suppose that $\xi$ is
only weakly spacelike (this is the case be Lemma \ref{prop:nonnegative mass}%
).\ Rotating coordinates, we consider the metric in the form $dx^{2}-dy^{2.}$
\ Take a basis for $\xi$, $\ \ $say $\{e_{i}\},$ and note that for each
$e_{i},$%
\[
dx^{2}(e_{i},e_{i})\geq dy^{2}(e_{i},e_{i})\geq0
\]
and $dx^{2}(e_{i},e_{i})$ is strictly positive for some $i.$ \ \ \ Now we make
a perturbation of each $e_{i}$ by scaling it slightly in the $x$ component.
\ The plane spanned by these vectors $\xi_{\varepsilon},$ is then strictly
spacelike. \ By continuity, we conclude that $\langle\eta,\xi\rangle<0$ for
some strictly spacelike plane, so has negative comass.
\end{proof}
\end{lemma}

\begin{claim}
The infimum in the definition of mass for the current $T$ is attained by the
calibration $\Phi.$
\end{claim}

\begin{proof}
Assume not. \ There exists a form $\varphi$ of comass $1$ with the property
that $T(\varphi)<T(\Phi).$ \ By the Riesz Representation%
\[
\int_{{\Gamma}}\langle\varphi(x)-\Phi(x),\xi(x)\rangle d\mathcal{H}^{n}<0
\]
so we conclude that there is a rectifiable set $\Upsilon\subset{\Gamma}~$ with
positive $\mathcal{H}^{n}$ measure such that on $\Upsilon$
\begin{equation}
\langle\varphi(x),\xi(x)\rangle<\langle\Phi(x),\xi(x)\rangle.
\label{to contradict, shortly}%
\end{equation}
\ If $\Upsilon$ projects to a set of positive measure in either direction
(i.e. either to $M$ or $\bar{M}$), we can conclude that it contains a point
where the optimal transport map is differentiable, and we can draw a
contradiction from Proposition \ref{prop:calibrated}: At points of
differentiability the tangent planes are calibrated, it follows that
\[
\langle\Phi(x),\xi(x)\rangle=||\xi||_{h^{\rho,\bar{\rho}}}\leq\langle
\varphi(x),\xi(x)\rangle
\]
because $\varphi$ has comass $1.$ So we conclude that $\Upsilon$ is in the
inverse image of sets of zero measure under both projections. It follows then
the calibrating form (which measures the volumes of projections of a given
plan in either direction) must vanish on $\xi$: so in fact we have
$\langle\varphi(x),\xi(x)\rangle<0.$ \ \ But this contradicts Lemma
\ref{comass negative}.
\end{proof}

\bigskip

This leads to our main theorem.

\begin{theorem}
\label{th:main}Suppose that $\rho,\bar{\rho}$ are smooth densities, and that
$c$ is smooth and satisfies (A1) (A2). \ Suppose the minimizing measure
$\gamma$ is compactly supported away from the cut locus. \ Then the spacelike
current $T$ defined on spt($\gamma)$ is homologically mass maximizing:
$\left\Vert T\right\Vert _{(N,h^{\rho,\bar{\rho}})}\geq\left\Vert S\right\Vert
_{(N,h^{\rho,\bar{\rho}})}$ \ for all compactly supported $n$-currents $S$
which are homologous to $T.$
\end{theorem}

\begin{proof}
\bigskip By the previous claim, we have $\left\Vert T\right\Vert
_{h^{\rho,\bar{\rho}}}=\Phi(T).$ \ \ If $S$ is homologous to $T$ then by
definition of mass (\ref{def:comass}) \ $\left\Vert S\right\Vert
_{h^{\rho,\bar{\rho}}}\leq\Phi(S)=\Phi(T),$ because $\Phi$ has comass $1$ by
Corollary, \ref{cor:cal}. \ 
\end{proof}

\bigskip

\ \ We note further that as calibration arguments require that we work with
currents of compact support, for many situations the result does not directly
apply globally (if $M$ is not compact), however it will apply locally. \ Note
that given any optimal map $F:M\rightarrow\bar{M}$ the restriction of the
optimal\ map to any subset is an optimal map onto its image. and satisfies the
same relations ~(\ref{eq:graph is spacelike}) (\ref{eq:graph is Lagrangian})
and (\ref{eq:measure preserving}), hence locally is calibrated.

We can now compare a smooth space-like submanifold to smooth variations of the
submanifold, and we conclude:

\begin{corollary}
\label{cor:mean curv} \ At any point where the graph of the optimal transport
map $F$ is a $C^{1}$ submanifold in $(M\times\bar{M}$,$h_{c}^{(\rho,\bar{\rho
})})$ the graph has zero mean curvature.
\end{corollary}

\begin{remark}
\label{rmk:curvature} We \ conclude with a remark on the Ma-Trudinger-Wang
cost-curvature condition \cite[A3 condition]{MTW}. \ The first two authors
\cite{KM} expressed the MTW\ condition as a curvature condition on $h_{c},$
which may be described as follows:\ \ \ The weak (strong) MTW
A3\ condition\ holds if and only if, {\ in each coordinate chart on $M$ and
$\bar{M}$, the Riemannian sectional curvature $R_{i\bar{j}\bar{j}i}$
corresponding to any vanishing component $(h_{c})_{i\bar{j}}=0$ of the metric
tensor is nonnegative (positive).}

With respect to the conformal metric $h^{\rho,\bar{\rho}}$, we have the
Riemann curvature tensor
\[
R_{i\bar{j}\bar{j}i}^{\rho,\bar{\rho}}=\big{(}\frac{\pi^{\ast}\rho\wedge
\bar{\pi}^{\ast}\bar{\rho}}{d\vol_{h_{c}}}\big{)}^{\frac{1}{n}}(R_{i\bar
{j}\bar{j}i}+\Lambda_{i\bar{j}}(h_{c})_{\bar{j}i}-\Lambda_{ii}(h_{c})_{\bar
{j}\bar{j}}+\Lambda_{\bar{j}i}(h_{c})_{i\bar{j}}-\Lambda_{\bar{j}\bar{j}%
}(h_{c})_{ii})
\]
for some $\Lambda_{ij\text{ }}$involving derivatives of the conformal factor.
\ $\ ${\ For metric components $(h_{c})_{i\bar{j}}=0$, we easily see the
corresponding component of the Riemann tensor} is given by
\[
R_{i\bar{j}\bar{j}i}^{\rho,\bar{\rho}}=\left(  \frac{\rho(x)\bar{\rho}%
(F(x))}{\left\vert \det(D\bar{D}c)\right\vert }\right)  ^{1/n}R_{i\bar{j}%
\bar{j}i}.
\]
Thus, the weak (strong) MTW A3\ condition\ holds if and only if, {in each
coordinate chart on $M$ and $\bar{M}$,} whenever $h_{i\bar{j}}{}^{\rho
,\bar{\rho}}=0,$ the sectional curvature $R_{i\bar{j}\bar{j}i}^{\rho,\bar
{\rho}}$ is nonnegative (positive).
\end{remark}

\end{document}